\documentclass[11pt,reqno,oneside]{amsproc}
\title[Lower bounds on blowup rate]{Lower bounds on the blowup rate of vorticity in the Euler equations}

\author[B.~Ingimarson]{Benjamin Ingimarson}
\address{Department of Mathematics, University of Southern California, Los Angeles, CA 90089}
\email{ingimars@usc.edu}

\author[I.~Kukavica]{Igor Kukavica}
\address{Department of Mathematics, University of Southern California, Los Angeles, CA 90089}
\email{kukavica@usc.edu}

  \chardef\forshowkeys=0
  \chardef\showllabel=0
  \chardef\refcheck=0
  \chardef\sketches=0
  \chardef\figures=1

\usepackage{enumitem,mathtools}
\usepackage{datetime}
\usepackage{fancyhdr}
\usepackage{comment}
\usepackage{listings}

\allowdisplaybreaks
\ifnum\forshowkeys=1

  \usepackage[notref,notcite,color]{showkeys}
\fi
\usepackage[margin=1in]{geometry}
\usepackage{amsmath, amsthm, amssymb}
\usepackage{times}
\usepackage{graphicx}
\usepackage[usenames,dvipsnames,svgnames,table]{xcolor}
\usepackage{marginnote}
\usepackage[unicode,breaklinks=true,colorlinks=true,linkcolor=blue,urlcolor=blue,citecolor=blue]{hyperref}
\usepackage[most]{tcolorbox}
\usepackage{tikz}

\ifnum\refcheck=1
  \usepackage{refcheck}
\fi

\usepackage{import}
\usepackage{xifthen}
\usepackage{pdfpages}
\usepackage{transparent}

\begin{document}
\def\YY{X}
\def\OO{\mathcal O}
\def\SS{\mathbb S}
\def\CC{\mathbb C}
\def\RR{\mathbb R}
\def\TT{\mathbb T}
\def\ZZ{\mathbb Z}
\def\HH{\mathbb H}
\def\RSZ{\mathcal R}
\def\LL{\mathcal L}
\def\SL{\LL^1}
\def\ZL{\LL^\infty}
\def\GG{\mathcal G}
\def\tt{\langle t\rangle}
\def\erf{\mathrm{Erf}}
\def\mgt#1{\textcolor{magenta}{#1}}
\def\ff{\rho}
\def\gg{G}
\def\sqrtnu{\sqrt{\nu}}
\def\ww{w}
\def\ft#1{#1_\xi}
\def\les{\lesssim}
\def\lec{\lesssim}
\def\ges{\gtrsim}
\def\gec{\gtrsim}
\renewcommand*{\Re}{\ensuremath{\mathrm{{\mathbb R}e\,}}}
\renewcommand*{\Im}{\ensuremath{\mathrm{{\mathbb I}m\,}}}
\ifnum\showllabel=1
 \def\llabel#1{\marginnote{\color{lightgray}\rm\small(#1)}[-0.0cm]\notag}
%%%%%%%%%%%%%%%%%%%\def\llabel#1{\label{#1}}
%  \reversemarginpar
  %\def\llabel#1{\notag}
%  \def\llabel#1{\label{#1}}
\else
% % \def\llabel#1{\nonumber}
 \def\llabel#1{\notag}
\fi
\newcommand{\norm}[1]{\left\|#1\right\|}
\newcommand{\nnorm}[1]{\lVert #1\rVert}
\newcommand{\abs}[1]{\left|#1\right|}
\newcommand{\NORM}[1]{|\!|\!| #1|\!|\!|}

\newtheorem{Theorem}{Theorem}[section]
\newtheorem{Corollary}[Theorem]{Corollary}
\newtheorem{Proposition}[Theorem]{Proposition}
\newtheorem{Lemma}[Theorem]{Lemma}
\newtheorem{Remark}[Theorem]{Remark}
\newtheorem{definition}{Definition}[section]

\def\theequation{\thesection.\arabic{equation}}
\numberwithin{equation}{section}
%%%%%%%%%%%%%%%%%%%%%%%%%%
\definecolor{mygray}{rgb}{.6,.6,.6}
\definecolor{myblue}{rgb}{9, 0, 1}
\definecolor{colorforkeys}{rgb}{1.0,0.0,0.0}
\newlength\mytemplen
\newsavebox\mytempbox
\makeatletter
\newcommand\mybluebox{%
    \@ifnextchar[%]
       {\@mybluebox}%
       {\@mybluebox[0pt]}}
\def\@mybluebox[#1]{%
    \@ifnextchar[%]
       {\@@mybluebox[#1]}%
       {\@@mybluebox[#1][0pt]}}
\def\@@mybluebox[#1][#2]#3{
    \sbox\mytempbox{#3}%
    \mytemplen\ht\mytempbox
    \advance\mytemplen #1\relax
    \ht\mytempbox\mytemplen
    \mytemplen\dp\mytempbox
    \advance\mytemplen #2\relax
    \dp\mytempbox\mytemplen
    \colorbox{myblue}{\hspace{1em}\usebox{\mytempbox}\hspace{1em}}}
\makeatother
%%%%%%%%%%%%%%%%%%%%%%%%%%

%Igor's macros  varmac
\def\pv{\text{p.v.}}
\def\bnew{\color{red}}
\def\enew{\color{black}}
\def\bold{\color{blue}}
\def\eold{\color{black}}
\def\rr{r}
\def\weaks{\text{\,\,\,\,\,\,weakly-* in }}
\def\inn{\text{\,\,\,\,\,\,in }}
\def\cof{\mathop{\rm cof\,}\nolimits}
\def\Dn{\frac{\partial}{\partial N}}
\def\Dnn#1{\frac{\partial #1}{\partial N}}
\def\tdb{\tilde{b}}
\def\tda{b}
\def\qqq{u}
\def\lat{\Delta_2}
\def\biglinem{\vskip0.5truecm\par==========================\par\vskip0.5truecm}
\def\inon#1{\hbox{\ \ \ \ \ \ \ }\hbox{#1}}                %in or on
\def\onon#1{\inon{on~$#1$}}
\def\inin#1{\inon{in~$#1$}}
\def\FF{F}
\def\andand{\text{\indeq and\indeq}}
\def\ww{w(y)}
\def\ll{{\color{red}\ell}}
\def\ee{\mathrm{e}}
\def\startnewsection#1#2{ \section{#1}\label{#2}\setcounter{equation}{0}}   
%\starts a new section\def\dist{\mathop{\rm dist}\nolimits}
\def\nnewpage{ }
\def\sgn{\mathop{\rm sgn\,}\nolimits}    
\def\Tr{\mathop{\rm Tr}\nolimits}    
\def\div{\mathop{\rm div}\nolimits}
\def\curl{\mathop{\rm curl}\nolimits}
\def\dist{\mathop{\rm dist}\nolimits}
\def\id{\mathop{\rm id}\nolimits}
\def\supp{\mathop{\rm supp}\nolimits}
\def\indeq{\quad{}}           
\def\period{.}                       
\def\semicolon{\,;}                  
\def\cmi#1{\text{~{{\coli IK: \underline{#1}}}~}}
\def\coli{\color{colorigor}}
\definecolor{colorigor}{rgb}{.5, 0.2, 0.8}
\def\colr{\color{red}}
\def\colrr{\color{black}}
\def\colb{\color{black}}
\def\coly{\color{lightgray}}
\definecolor{colorgggg}{rgb}{0.1,0.5,0.3}
\definecolor{colorllll}{rgb}{0.0,0.7,0.0}
\definecolor{colorhhhh}{rgb}{0.3,0.75,0.4}
\definecolor{colorpppp}{rgb}{0.7,0.0,0.2}
%\definecolor{coloroooo}{rgb}{0.9,0.4,0}
\definecolor{coloroooo}{rgb}{0.45,0.0,0.0}
\definecolor{colorqqqq}{rgb}{0.1,0.7,0}
\def\colg{\color{colorgggg}}
\def\collg{\color{colorllll}}
\def\cole{\color{coloroooo}}
\def\coleo{\color{colorpppp}}
\def\colu{\color{blue}}
\def\colc{\color{colorhhhh}}
\def\colW{\colb}   %color for weight
\definecolor{coloraaaa}{rgb}{0.6,0.6,0.6}%%%out
\def\colw{\color{coloraaaa}}
\def\comma{ {\rm ,\qquad{}} }            
\def\commaone{ {\rm ,\quad{}} }          
\def\les{\lesssim}
\def\nts#1{{\color{red}\hbox{\bf ~#1~}}} 
\def\ntsf#1{\footnote{\color{colorgggg}\hbox{#1}}} 
\def\blackdot{{\color{red}{\hskip-.0truecm\rule[-1mm]{4mm}{4mm}\hskip.2truecm}}\hskip-.3truecm}
\def\bluedot{{\color{blue}{\hskip-.0truecm\rule[-1mm]{4mm}{4mm}\hskip.2truecm}}\hskip-.3truecm}
\def\purpledot{{\color{colorpppp}{\hskip-.0truecm\rule[-1mm]{4mm}{4mm}\hskip.2truecm}}\hskip-.3truecm}
\def\greendot{{\color{colorgggg}{\hskip-.0truecm\rule[-1mm]{4mm}{4mm}\hskip.2truecm}}\hskip-.3truecm}
\def\cyandot{{\color{cyan}{\hskip-.0truecm\rule[-1mm]{4mm}{4mm}\hskip.2truecm}}\hskip-.3truecm}
\def\reddot{{\color{red}{\hskip-.0truecm\rule[-1mm]{4mm}{4mm}\hskip.2truecm}}\hskip-.3truecm}
\def\tdot{{\color{green}{\hskip-.0truecm\rule[-.5mm]{3mm}{3mm}\hskip.2truecm}}\hskip-.1truecm}
\def\gdot{\greendot}
\def\bdot{\bluedot}
\def\ydot{\cyandot}
\def\rdot{\cyandot}
\def\fractext#1#2{{#1}/{#2}}
\def\ii{\hat\imath}
%%%%%%%
\def\fei#1{\textcolor{blue}{#1}}
\def\vlad#1{\textcolor{cyan}{#1}}
\def\igor#1{\text{{\textcolor{colorqqqq}{#1}}}}
\def\igorf#1{\footnote{\text{{\textcolor{colorqqqq}{#1}}}}}
%varmac
%\def\AA{Y}
\def\Omf{\Omega_{\text f}}
\def\Ome{\Omega_{\text e}}
\def\Omb{\Omega_{\text b}}
\def\Gaf{\Gamma_{\text f}}
\def\Gae{\Gamma_{\text e}}
\def\Gab{\Gamma_{\text b}}
\def\Gac{\Gamma_{\text c}}
\def\Nf{N^{\text f}}
\def\Ne{N^{\text e}}

\newcommand{\p}{\partial}
\renewcommand{\d}{\mathrm{d}}
\newcommand{\UE}{U^{\rm E}}
\newcommand{\PE}{P^{\rm E}}
\newcommand{\KP}{K_{\rm P}}
\newcommand{\uNS}{u^{\rm NS}}
\newcommand{\vNS}{v^{\rm NS}}
\newcommand{\pNS}{p^{\rm NS}}
\newcommand{\omegaNS}{\omega^{\rm NS}}
\newcommand{\uE}{u^{\rm E}}
\newcommand{\vE}{v^{\rm E}}
\newcommand{\pE}{p^{\rm E}}
\newcommand{\omegaE}{\omega^{\rm E}}
\newcommand{\ua}{u_{\rm   a}}
\newcommand{\va}{v_{\rm   a}}
\newcommand{\omegaa}{\omega_{\rm   a}}
\newcommand{\ue}{u_{\rm   e}}
\newcommand{\ve}{v_{\rm   e}}
\newcommand{\omegae}{\omega_{\rm e}}
\newcommand{\omegaeic}{\omega_{{\rm e}0}}
\newcommand{\ueic}{u_{{\rm   e}0}}
\newcommand{\veic}{v_{{\rm   e}0}}
\newcommand{\up}{u^{\rm P}}
\newcommand{\vp}{v^{\rm P}}
\newcommand{\tup}{{\tilde u}^{\rm P}}
\newcommand{\bvp}{{\bar v}^{\rm P}}
\newcommand{\omegap}{\omega^{\rm P}}
\newcommand{\tomegap}{\tilde \omega^{\rm P}}
\newcommand{\eps}{\varepsilon}  %ben's macro
\newcommand{\eqnb}{\begin{equation}}
\newcommand{\eqne}{\end{equation}}
  
\renewcommand{\up}{u^{\rm P}}
\renewcommand{\vp}{v^{\rm P}}
\renewcommand{\omegap}{\Omega^{\rm P}}
\renewcommand{\tomegap}{\omega^{\rm P}}

\begin{abstract}
Under the assumption that a solution to the 3D incompressible Euler equations blows up at a time $T_\ast$ and that $T_\ast $ is the first such time, we establish lower bounds on the rate of blow-up of the maximum norm of the vorticity.
In particular, when the domain is $\mathbb{R}^3$ or $\mathbb{T}^3$, we
provide lower bounds on $\int_{0}^{t}\Vert \omega\Vert_{L^\infty}\,ds$ and
$\sup_{s\in[0,t]}\|\omega\|_{L^\infty}$ for $t$ sufficiently close to~$T_\ast$.
Notably, this gives a quantitative description of the BKM blow-up criterion.
Moreover, we provide pointwise-in-time lower bounds on~$\|D^k \omega\|_{L^\infty}$.
Finally, we state some consequences on the blow-up rate of the derivative of the deformation tensor. 
%Under the assumption that a solution to the 3D incompressible Euler equations blows up at a time $T_\ast$ and that $T_\ast $ is the first such time, we establish lower bounds on the rate of blow-up of the maximum norm of the vorticity.
%In particular, when the domain is $\mathbb{R}^3$ or $\mathbb{T}^3$, we provide lower bounds on the accumulation of $\|\omega\|_{L^\infty}$ up to time $t$ and the supremum over $[0,t]$ of $\|\omega\|_{L^\infty}$ for $t$ sufficiently close to~$T_\ast$.
%Notably, this gives a quantitative description of the BKM blow-up criterion.
%Moreover, we provide pointwise-in-time lower bounds on $\|D^k \omega\|_{L^\infty}$.
%Finally, we state some consequences on the blow-up rate of the derivative of the deformation tensor. 
%\eold
\end{abstract}

\maketitle

\setcounter{tocdepth}{2} 
\section{Introduction}\label{sec00}
We consider the 3D incompressible Euler equations, 
  \begin{align}
   \begin{split}
    u_t
    + u \cdot \nabla u
    + \nabla p
    &= 0
    ,
    \\
    \div u
    &= 0
    ,
   \end{split}
   \label{EQ01}
  \end{align}
in $\Omega \times [0,T]$ for the case that $\Omega = \mathbb{R}^3$ or~$\mathbb{T}^3$.
The classical local existence theorem, see~\cite{EM,K,T}, states that for an initial velocity field $u_0 \in H^r$, with $r > 2.5$, where $\|u_0\|_{H^r} \leq M$ for some $M > 0$, there exists a time $T = T(M) > 0$ such that~\eqref{EQ01} has a solution in the class,
  \begin{equation}
   u 
   \in
   C([0,T]; H^r(\Omega)) \cap C^1 ([0,T]; H^{r-1}(\Omega))
   .
   \label{EQaa}
  \end{equation}
A classical result by Beale, Kato, and Majda~\cite{BKM} asserts that if such a solution leaves the class~\eqref{EQaa} at a time $T_\ast$ and that $T_\ast$ is the first such time, then the spatial $L^\infty$ norm of the vorticity blows up as~$t\to T_\ast^-$.
More precisely, 
  \begin{equation}
   \int_0^{T_\ast} \|\omega(t)\|_{L^\infty}
   = \infty
   ,
   \label{EQ03}
  \end{equation}
and, consequently,
  \begin{equation}
   \limsup_{t \to T_\ast^-} \|\omega(t)\|_{L^\infty}
   = \infty
   ,
   \label{EQcc}
  \end{equation}
where $\omega = \curl u$ denotes the vorticity, see also \cite[Theorem~2.6]{BV} and \cite[Theorem~4.3]{MB}.
In other words, the solution $u$ stays smooth provided the maximum norm of the vorticity is time-integrable. 

However, not much is known regarding the rate at which either of the two quantities~\eqref{EQ03} or~\eqref{EQcc} blow up. 
In this paper, we prove lower bounds on the blow up rate of the accumulation of vorticity, $\int_0^t \|\omega(s)\|_{L^\infty}$, and $\sup_{s \leq t} \|\omega(s)\|_{L^\infty }$, for $t$ sufficiently close to~$T_\ast$. 
We also provide \textit{pointwise-in-time} lower bounds on the higher-order derivatives,~$\|D^k\omega(s)\|_{L^\infty}$.
We end the paper with some consequences on the blow-up of $\|D\sigma\|_{L^\infty}$, where $\sigma$ denotes the deformation tensor.

In previous literature, several variants of the BKM criterion have been shown to hold. 
Indeed, Ferrari~\cite{F} proved the criterion~\eqref{EQ03} in the case of a bounded domain, see also~\cite{SY,Z}.
As for different quantities in place of the maximum norm of the vorticity, Ponce~\cite{P} proved that it can be replaced with the maximum norm of the deformation tensor. 
Then, Constantin, Fefferman, and Majda~\cite{CFM} established a blow-up criterion involving the vorticity direction and the maximum norm of the velocity. 
In a result which directly improves the BKM criterion, Kozono and Taniuchi~\cite{KT} demonstrated that the $L^\infty$ norm in~\eqref{EQ03} can be replaced by the BMO norm.
With Ogawa~\cite{KOT}, they further improved this result by replacing the BMO norm with~$\dot{B}_{\infty,\infty}^0$.
Finally, we note that Aydin~\cite{A} established that the breakdown of a solution $u$ is contingent on the $L^2$ integrability of only its tangential derivatives. 

These results are, of course, a~priori in nature; whether or not solutions to~\eqref{EQ01} with smooth initial data blow up in finite time remains an open problem.
However, under assumptions with initial data in lower regularity spaces, there have been major advances in recent years. 
Elgindi~\cite{E} established finite-time singularity formation in the axisymmetric no-swirl class for initial data in $C^{1,\alpha}$ for sufficiently small $\alpha > 0$.
We also note the work by C\'ordoba, Martinez-Zoroa, and Yang~\cite{CMZ} in which they established finite-time blow-up for a class of initial data in~$C^\infty (\mathbb{R}^3 \setminus \{0\}) \cap C^{1,\alpha} \cap L^2 $, with $\alpha > 0$ sufficiently small. 
Recently, Shkoller~\cite{S} obtained a threshold result by showing finite-time Type-I blow-up in the axisymmetric no-swirl class with initial data in $C^{1,1/3-}\cap L^2$,
constructing solutions satisfying $\|\omega(t)\|_{L^\infty} \sim (T_\ast -t)^{-1}$ as~$t \to T_\ast^-$.
Finally, we note the paper by Jeong, Mart\'inez-Zoroa, and O\.{z}a\'{n}ski~\cite{JMO}, in which they proved instantaneous and continuous-in-time loss of supercritical Sobolev regularity for~\eqref{EQ01}. 

Assuming a solution blows up at a time $T_\ast$, a natural question concerns the rate at which the singularity forms.
Under the assumption of blow-up, one can prove a lower bound on the blow-up rate if the quantity satisfies, say, a Riccati-type inequality of the form, $\dot{x} \les x^\beta$, where $\beta > 1$. 
For a solution $u$ to the 3D incompressible Euler equations which leaves the class~\eqref{EQaa} at a time $T_\ast$, Chen and Pavlovi\'c~\cite{CP} proved  pointwise lower bounds on $u$ in the $H^r$ norm for $t$ sufficiently close to~$T_\ast$
of the order $(T_\ast - t)^{-2\delta/5-1}$.

Recently, there has been interest in the rate of blow-up of the vorticity itself.
Motivated by solutions with self-similar blow-up, Constantin, Ignatova, and Vicol~\cite{CIV} established a criterion for the rate of blow-up of~$\|D\omega\|_{L^\infty}$.
Assuming $\omega_0 \in C^1$ and that a solution to~\eqref{EQ01} with finite initial kinetic energy blows up at a time $T_\ast$, they proved
  \begin{equation}
   \sup_{t \in[0,T_\ast)} (T_\ast-t)^{1+\gamma} \|D \omega\|_{L^\infty(\mathbb{R}^3)}
   =
   \infty
   \label{EQdd}
  \end{equation}
for all $\gamma < 2/5$. 

In this paper, we improve on the classical BKM criterion with explicit lower bounds on the blow up rate of the accumulation of vorticity, $\int_0^t \|\omega(s)\|_{L^\infty}$, and non-integrable lower bounds on the quantity, $\sup_{s \leq t} \|\omega(s)\|_{L^\infty}$, for $t$ sufficiently close to~$T_\ast$.
Moreover, we
obtain the lower bound \eqref{EQdd} with $\gamma=2/5$.
Also, for any $k\in\mathbb{N}$, we provide pointwise lower bounds on the quantity~$\|D^k \omega(t)\|_{L^\infty} $.
Finally, we detail some consequences on the blow-up rate of the deformation tensor.
These results are detailed in Section~\ref{sec01}.

In Section~\ref{sec02}, we prove Theorem~\ref{T01}, which provides lower bounds on~$\omega$. 
To establish the estimate on the accumulation of vorticity, we first derive a Gronwall-type inequality of the form~\eqref{EQ111} for $\|\omega(t)\|_{L^\infty}$ on the interval $[0,T_\ast)$. 
However, it turns out that the differential inequality 
contains an exponential of the accumulation $\int_0^t \|\omega\|_{L^\infty}$, which
does not seem amenable to a barrier argument. To bypass this, we 
parameterize time in terms of the accumulation of the vorticity, called~$\tau$, see~\eqref{EQ112} below. 
We then reformulate our exponential inequality for $(d/dt) \|\omega(t)\|_{L^\infty}$ in terms of the new parameter.
The lower bound on the accumulation is then obtained after considering~$dt/d\tau$. 

In Section~\ref{sec03}, we prove Theorems~\ref{T02},~\ref{T03}, and~\ref{T04}.
Theorem~\ref{T02} provides pointwise-in-time lower bounds on the blow-up of~$\|D\omega\|_{L^\infty}$.
The proof starts with a lemma stating an interpolation inequality relating $\|D^{1/2}\omega\|_{L^8}$ and~$\|D\omega\|_{L^\infty}$.
We then derive a Riccati-type inequality for the $L^8$ norm,
from which we deduce a lower bound on~$\|D\omega\|_{L^\infty}$.
We note that this lower bound is at the critical exponent $\gamma = 2/5$. 
This result implies that the singularity for $\|D\omega\|_{L^\infty}$ develops as a pointwise limit to infinity. 
In Theorem~\ref{T03}, we demonstrate that Theorem~\ref{T02} may be extended to analogous pointwise-in-time lower bounds for~$\|D^k\omega\|_{L^\infty}$.
Finally, in Theorem~\ref{T04}, we provide lower bounds on the uniform norm of
the derivative of the stress $\sigma=\nabla u+ (\nabla u)^{\text T}$
considered by Ponce in~\cite{P}.

\section{Main results}\label{sec01}
The first theorem provides lower bounds on the blow-up rate of $\omega$ when $\Omega = \mathbb{R}^3$ or~$\mathbb{T}^3$. 

\begin{Theorem}
\label{T01}
Let $u$ be a solution to the Euler equations~\eqref{EQ01} for $\Omega = \mathbb{R}^3$ or $\mathbb{T}^3$ in the regularity class~\eqref{EQaa} with $r = 3$. 
Suppose there is a time $T_\ast>0$ such that the solution cannot be continued in this class to $T= T_\ast$ and that $T_\ast$ is the first such time. 
Then, we have 
  \begin{equation}
   \int_0^{t} \|\omega(s) \|_{L^\infty}\, ds
   \geq 
   \frac{1}{{C}}\log\log \frac{1}{T_\ast -t}
   ,
   \quad\quad
   T_\ast -t 
   <
   \frac{1}{C}
   ,
   \label{EQ101}
  \end{equation}
and, moreover, 
  \begin{equation}
   \limsup_{t \to T_\ast^-}\, (T_\ast -t) \log\left( \frac{1}{T_\ast - t}\right) \|\omega(t)\|_{L^\infty}
   \geq
   \frac{1}{C}
   ,
   \label{EQ101a}
  \end{equation}
where $C$ denotes a constant depending on~$\|u_0\|_{H^3}$.  
\end{Theorem}

\begin{Remark}
\label{R00}
{\rm
We note that the proof shows that the constant in $C$ can be made arbitrarily close to the
constant in the inequality \eqref{EQ105} below.
}
\end{Remark}

\begin{Remark}
\label{R01}
{\rm
We also note that Theorem~\ref{T01} holds for solutions in the regularity class~\eqref{EQaa} with $r > 5/ 2$, and the constant appearing in~\eqref{EQ105} depends on~$r$.
}
\end{Remark}

\begin{Remark}
\label{R03}
{\rm
The statement and proof apply directly to the case of a bounded smooth domain with impermeable boundary condition, using
the inequality for $\Vert \nabla u\Vert_{L^{\infty}(\Omega)}$ in~\cite{F,SY}.
}
\end{Remark}

The second theorem gives pointwise-in-time lower bounds on~$\|D\omega\|_{L^\infty}$. 

\begin{Theorem}
\label{T02}
Assume the hypothesis of Theorem~\ref{T01} for $\Omega = \mathbb{R}^3$  and $r > 7/2$.
Then, we have
  \begin{equation}
   \|D\omega(t)\|_{L^\infty}
   \geq
   \frac{1}{C(T_\ast -t)_{}^{7/5}}
   \comma
   t \in [0,T_\ast)
   ,
   \label{EQ101c}
  \end{equation}
where $C$ depends on~$\|u_0\|_{L^2}$.
When $\Omega = \mathbb{T}^3$, the lower bound~\eqref{EQ101c} holds for $t$ sufficiently close to~$T_\ast$. 
\end{Theorem}

We also give pointwise-in-time lower bounds on the blow-up of $\|D^k \omega\|_{L^\infty}$ for $k \geq 2$.

\begin{Theorem}
\label{T03}
Assume the hypothesis of Theorem~\ref{T01} for $\Omega = \mathbb{R}^3$ and $r > 5/2 + m$, 
where $m\in\{3,4,\ldots\}$. Then, we have
  \begin{equation}
    \|D^k \omega(t)\|_{L^\infty} 
    \geq 
    \frac{1}{C(T_\ast -t)^{2k/5+1}}
    ,
    \quad\quad
    t \in [0,T_\ast)
    \label{EQ07a}
  \end{equation}
for $k = 2,\dots, m$, where $C$ depends on~$\|u_0\|_{L^2}$. 
When $\Omega = \mathbb{T}^3$, the lower bound~\eqref{EQ07a} holds for $t$ sufficiently close to~$T_\ast$.
\end{Theorem}

Finally, we state the blow-up rate
for the stress tensor $\sigma=\nabla u+ (\nabla u)^{\text T}$
addressed in~\cite{P}.

\begin{Theorem}
\label{T04}
Assume the hypothesis of Theorem~\ref{T01} for $\Omega = \mathbb{R}^3$ or $\mathbb{T}^3$ for $r > 7/2$.
Then, we have
  \begin{equation}
   \|D\sigma(t)\|_{L^\infty}
   \geq
   \frac{1}{C(T_\ast -t)_{}^{7/5} }
   \llabel{EQ07}
  \end{equation}
for all $t$ sufficiently close to $T_\ast$,
where $C$ depends on~$\Vert u_0\Vert_{L^2}$.
\end{Theorem}

Analogous bounds to \eqref{EQ07a} also apply for
$\|D^k \sigma(t)\|_{L^\infty}$ for any $k\ge2$.
%, using the idea of the proof of Theorem~\ref{T04}

\section{Lower bounds on the blow-up rate of $\|\omega\|_{L^\infty}$}\label{sec02}
In this section we prove Theorem~\ref{T01}. 
\noindent 

\begin{proof}[Proof of Theorem~\ref{T01}]
Let $t \in [0,T_\ast)$. 
We begin with the vorticity equation corresponding to~\eqref{EQ01},
  \begin{equation}
   \partial_t \omega
   + u\cdot \nabla \omega
   =
   \omega \cdot \nabla u
   .
   \label{EQ103}
  \end{equation}
Following Lagrangian trajectories, we obtain 
\begin{equation}
   \frac{d}{dt}\|\omega\|_{L^\infty}
   \les 
   \|D u\|_{L^\infty} \|\omega\|_{L^\infty}
   .
   \label{EQ104}
  \end{equation}
In conjunction with the bound, 
  \begin{equation}
   \|D u \|_{L^\infty} 
   \leq
   \widetilde{C}(1 + \|\omega\|_{L^\infty}(1+ \log_+\|u\|_{H^3}))
   ,
   \label{EQ105}
  \end{equation}
see~\cite{KT},
the estimate~\eqref{EQ104} becomes
  \begin{equation}
   \frac{d}{dt} \|\omega\|_{L^\infty}
   \les
   \|\omega\|_{L^\infty}
   + \|\omega\|_{L^\infty}^2 (1 + \log_+\|u\|_{H^3})
   ,
   \label{EQ106}
  \end{equation}
where we denote $\log_+ a = \max\{\log a,0\}$ for $a>0$.
Next, we bound $\|u\|_{H^3}$ in terms of~$\|\omega\|_{L^\infty}$. 
From~\eqref{EQ01}, we have
  \begin{equation}
   \frac{d}{dt} \|u\|_{H^3}
   \les
   \|Du\|_{L^\infty}\|u\|_{H^3}
   \les 
   (1 + \|\omega\|_{L^\infty} (1 + \log_+\|u\|_{H^3}))\|u\|_{H^3}
   .
   \llabel{EQ107}
  \end{equation}
Hence,
  \begin{equation}
   \left(\frac{d}{dt}\right)^+ (1 + \log_+ \|u\|_{H^3})
   \les
   1 + \|\omega\|_{L^\infty} (1 + \log_+ \|u\|_{H^3})
   ,
   \llabel{EQ108}
  \end{equation}
where $(d/dt)^+$ denotes the upper right derivative. 
Thus, Gronwall's inequality gives 
  \begin{equation}
   1 + \log_+\|u(t)\|_{H^3}
   \les
   \exp\left(\widetilde{C}\int_0^t \|\omega(s)\|_{L^\infty} \, ds\right)
    \comma t<T_*
    ,
   \llabel{EQ109}
  \end{equation}
where $\widetilde{C}$ denotes the constant from~\eqref{EQ105}.  
Substituting this into~\eqref{EQ106}, we arrive at 
  \begin{equation}
   \frac{d}{dt}\|\omega\|_{L^\infty}
   \les
   \|\omega\|_{L^\infty}
   + \|\omega\|_{L^\infty}^2 \exp\left(\widetilde{C}\int_0^t \|\omega\|_{L^\infty}\, ds\right)
   ,
   \inon{$t< T_\ast$,}
   \label{EQ110}
  \end{equation}
where the implicit constant depends on $\|u_0\|_{H^3}$ and~$T_\ast$. 

For brevity, we denote $x(t) = \|\omega(t) \|_{L^\infty}$ and rewrite~\eqref{EQ110} as
  \begin{align}
   \begin{split}
   \dot{x} 
   &\les
   x + x^2 \exp\left(\int_0^t x \, ds \right)
   ,
   \\
   x(0) &= \|\omega(0)\|_{L^\infty}
   ,
   \end{split}
   \label{EQ111}
  \end{align}
  where we assume $\widetilde{C} =1$ with no loss of generality, 
otherwise, we can include it in the following reparameterization~\eqref{EQ112}.
Next, we convert~\eqref{EQ111} into an inequality in terms of the accumulation of vorticity, which we denote as
  \begin{equation}
   \tau(t)
   =
   \int_0^t x(s)\, ds
   ,
   \inon{$t < T_\ast$.}
   \label{EQ112}
  \end{equation}
Since $x(t) = \|\omega(t)\|_{L^\infty}$ has positive initial data, we have $x(t) > 0$ for all $t \in [0,T_\ast)$. 
Hence, $\tau\colon[0,T_\ast)\to [0,\infty)$ is a $C^1$ diffeomorphism. 
This enables us to parameterize the time as $t = t(\tau)$, where
  \begin{equation}
   \lim_{\tau \to \infty} t(\tau) = T_\ast
   \andand
   \lim_{\tau \to 0^+} t(\tau) = 0
   .
   \label{EQ116b}
  \end{equation}
Denoting $X(\tau) = x(t(\tau))$,
we have
  \begin{equation}
   \frac{d\tau}{dt} =  x(t)
   \andand
   \frac{dt}{d\tau} = \frac{1}{X(\tau)}
   \llabel{EQ114}
  \end{equation}
with $X(0) = x(0)$.
Therefore, the differential inequality~\eqref{EQ111} can be recast as
  \begin{equation}
   \frac{d X}{d\tau} 
   \les
   1 + X e^\tau
   ,
   \llabel{EQ115}
  \end{equation}
and applying Gronwall's inequality on the interval $[0,\tau]$ gives
  \begin{equation}
   X(\tau) 
   \leq 
   C(1+ \tau) e^{Ce^\tau}
   ,
   \inon{$\tau \geq 0$}
   ,
   \llabel{EQ116}
  \end{equation}
where $C \geq 1$ denotes a constant 
depending on $\|u_0\|_{H^3}$ and~$T_\ast$.
Next, we recover an explicit relationship between the quantity $T_\ast -t$ and the accumulation~$\tau$.
For $\tau_0 \geq 0$, we denote $t_0 = t(\tau_0)$. 
Recalling~\eqref{EQ116b}, we have 
  \begin{equation}
   T_\ast - t_0
   = 
   \int_{\tau_0}^\infty \frac{dt}{d\tau} \, d\tau
   =
   \int_{\tau_0}^\infty \frac{1}{X(\tau)}\, d\tau
   \geq
   C^{-1}\int_{\tau_0}^\infty (1+\tau)^{-1}e^{-Ce^{\tau}} \, d\tau
   \geq
   \frac{1}{C(1+\tau_0) e^{Ce^{\tau_0 + 1}}}
   .
   \llabel{EQ117}
  \end{equation}
To obtain the last inequality, we estimate the integral bounds with $[\tau_0,\tau_0+1]$
and bound the integrand from below with its minimum on $[\tau_0,\tau_0+1]$.
Next, we drop the subscripts, invert the inequality, and bound $2+ \tau \leq e^{e^{\tau+1}}$ for $\tau \geq 0$, arriving at
\begin{equation}
   e^{(1+C)e^{\tau+1}}
   \geq 
   \frac{1}{C(T_\ast - t)}
   ,
   \inon{$\tau \geq 0$.}
   \llabel{EQ118}
  \end{equation}
We then apply a logarithm, which leads to
  \begin{equation}
   (1+C)e^{\tau+1}
   \geq
   \log\left( \frac{1}{C(T_\ast -t)}\right)
   .
   \llabel{EQ119}
  \end{equation}
Taking a second logarithm gives
  \begin{equation}
   \tau 
   \geq 
   \log\log \left( \frac{1}{C(T_\ast -t)} \right)
   -\log(1+C)
   -1
   ,
   \quad\quad
   T_\ast - t
   \leq 
   \frac{1}{C}
   .
   \llabel{EQ120}
  \end{equation}
Thus, we have obtained
\begin{equation}
   \int_0^{t} \|\omega(s) \|_{L^\infty}\, ds
   \geq 
   \frac{1}{{C_1}}\log\log\left( \frac{1}{C_2(T_\ast -t)}\right)
   - C_3
   ,
   \quad\quad
   T_\ast -t 
   <
   \frac{1}{C_4}
   ,
   \label{EQ02}
  \end{equation}
where $C_1, C_2, C_3, C_4$ are positive constants.
The right-hand side of \eqref{EQ02}
may be rewritten as
  \begin{equation}
   \frac{1}{C_1}
   \log \log \frac{1}{T_*-t}
   +
   \frac{1}{C_1}
   \log\left(1+ \frac{\log C_2}{\log (1/(T_*-t))}\right)
   .
   \llabel{EQ04}
  \end{equation}
Note that the second term converges to $0$ as $t\to T_{*}$, and thus,
by possibly increasing $C_4$, we derive
\begin{equation}
   \int_0^{t} \|\omega(s) \|_{L^\infty}\, ds
   \geq 
   \frac{1}{{C_1}}\log\log\left( \frac{1}{T_\ast -t}\right)
   - C_3
   ,
   \quad\quad
   T_\ast -t 
   <
   \frac{1}{C_4}
   .
   \llabel{EQ05}
  \end{equation}
Finally, by increasing $C_1$ and $C_4$, we may also,
without loss of generality, set $C_3=0$, and 
\eqref{EQ101} is proven.

We now prove~\eqref{EQ101a}. 
Towards a contradiction, we suppose
  \begin{equation}
   \limsup_{t\to T_\ast^-}\, (T_\ast -t) \log\left( \frac{1}{T_\ast -t}\right) \|\omega(t)\|_{L^\infty} 
   \leq
   \delta
   ,
   \label{EQ122}
  \end{equation}
for some $\delta \in [0,1)$
and claim that this leads to a contradiction if $\delta$ is sufficiently small.
Denote $y(t) = \sup_{s\leq t} \|\omega(s)\|_{L^\infty}$.
By monotonicity, the equation~\eqref{EQ122} implies
  \begin{equation}
   y(t)
   \leq \frac{2\delta}{(T_\ast -t) \log \left( \frac{1}{T_\ast -t}\right)},
   \inon{$t_0 \leq t < T_\ast$,}
   \llabel{EQ123}
  \end{equation}
where $t_0$ is a fixed number chosen such that~\eqref{EQ101} holds for all $t \geq t_0$.
Thus,
  \begin{align}
   \begin{split}
    \int_0^t y\, ds
    &\leq
    \int_0^{t_0} y\, ds
    +
    2\delta\int_{t_0}^t \frac{1}{(T_\ast -s) \log \left(\frac{1}{T_\ast -s}\right)}\,ds
    \\&=
    \int_0^{t_0} y\, ds
    +2\delta \log\log \left( \frac{1}{T_\ast -t}\right)    
    - 2\delta \log\log\left( \frac{1}{T_\ast-t_0}\right)
    .
   \label{EQ125}
   \end{split}
  \end{align}
However, from~\eqref{EQ101}, we have the lower bound,
  \begin{equation}
   \int_0^t y\, ds
   \geq 
   \frac{1}{C}
   \log\log \left( \frac{1}{T_\ast -t}\right)
   .
   \label{EQ124}
  \end{equation}
Upon combining~\eqref{EQ125} and~\eqref{EQ124}, we infer
  \begin{equation}
   \left(
    \frac{1}{C} - 2\delta
   \right)
   \log \log \frac{1}{T_*-t}
   \leq
   \int_{0}^{t_0} y\,ds
   - 2\delta \log \log \frac{1}{T_*-t_0}
   ,
   \llabel{EQ06}
  \end{equation}
for $t$ sufficiently close to $T_*$, which is a contradiction if
$\delta>0$ is sufficiently small.
\end{proof}

\section{Lower bounds on the blow-up rate of $\|D^k\omega\|_{L^\infty}$}\label{sec03}
First, we prove lower bounds on~$\|D\omega\|_{L^\infty}$.
The argument relies on the following lemma. 

\begin{Lemma}
\label{L01}
Assume the hypothesis of Theorem~\ref{T02}. 
Then, in the case that $\Omega = \mathbb{R}^3$,
  \begin{equation}
   \|D^{1/2} \omega\|_{L^8} 
   \les
   \|u\|_{L^2}^{1/4}\|D\omega\|_{L^\infty}^{3/4}
   ,
   \label{EQ300}
  \end{equation}
while if $\Omega = \mathbb{T}^3$, we have
  \begin{equation}
   \|D^{1/2} \omega\|_{L^8}
   \les
   \|u\|_{L^2}^{1/4}\|D \omega\|_{L^\infty}^{3/4} 
   + \|u\|_{L^2}
   .
   \label{EQ301}
  \end{equation}
\end{Lemma}

\begin{proof}
First, we consider the case $\Omega = \mathbb{R}^3$. 
Using the fractional Gagliardo-Nirenberg inequalities~\cite[Theorem~1]{BM}, we get
  \begin{equation}
   \|\omega\|_{W^{1/2,8}}
   \les
   \|\omega\|_{L^4}^{1/2} \|\omega\|_{W^{1,\infty}}^{1/2}
   .
   \label{EQ302}
  \end{equation}
To estimate the first term on the right-hand side, we have
  \begin{equation}
    \|\omega\|_{L^4}
    \les 
    \|u\|_{W^{1,4}}
    \les
    \|u\|_{W^{1/2,8/3}}^{1/2} \|u\|_{W^{3/2,8}}^{1/2}
    \les
    \|u\|_{W^{1/2,8/3}}^{1/2} \|\omega\|_{W^{1/2,8}}^{1/2}
    + \|u\|_{W^{1/2,8/3}}^{1/2} \|u\|_{L^8}^{1/2}
    ,
   \label{EQ303}
  \end{equation}
where, after the third inequality,
we used the Calder\'on-Zygmund theorem, see~\cite{St}. 
After rescaling,~\eqref{EQ303} gives 
  \begin{equation}
   \|\omega\|_{L^4}
   \les 
   \|D^{1/2}u\|_{L^{8/3}}^{1/2} \|D^{1/2}\omega\|_{L^{8}}^{1/2}
   .
   \label{EQ304}
  \end{equation}
Again, by the fractional Gagliardo-Nirenberg inequality, we have
  \begin{equation}
   \|u\|_{W^{1/2,8/3}}
   \les
   \|u\|_{L^2}^{1/2} \|u\|_{W^{1,4}}^{1/2}
   \les
   \|u\|_{L^2}^{1/2} \|\omega\|_{L^4}^{1/2}
   + \|u\|_{L^2}^{1/2} \|u\|_{L^4}^{1/2}
   ,
   \label{EQ305}
  \end{equation}
and, in particular,
  \begin{equation}
   \|D^{1/2}u\|_{L^{8/3}} 
   \les
   \|u\|_{L^2}^{1/2}
   \|\omega\|_{L^4}^{1/2}
   .
   \label{EQ306}
  \end{equation}
Substituting \eqref{EQ306} into~\eqref{EQ304}, we get
  \begin{equation}
   \|\omega\|_{L^4}
   \les
   \|u\|_{L^2}^{1/3} \|D^{1/2}\omega\|_{L^{8}}^{2/3}
   .
   \label{EQ307}
  \end{equation}
Furthermore, replacing~\eqref{EQ307} into~\eqref{EQ302} gives us
  \begin{equation}
   \|D^{1/2}\omega\|_{L^8}
   \les
   \|u\|_{L^2}^{1/4}\|\omega\|_{W^{1,\infty}}^{3/4}
   .
   \label{EQ308}
  \end{equation}
After rescaling to eliminate the lower order term in~\eqref{EQ308}, we obtain~\eqref{EQ300}.
Also, by considering the projection onto divergence-free vector fields, the estimate~\eqref{EQ300} is valid for $\tilde{\omega} = \curl \tilde{u}$ where $\tilde{u}$ is not divergence free.

Now, we consider the case $\Omega = \mathbb{T}^3$.
Let $\varphi$ be a smooth cut-off function equal to $1$ on the cell $[0,1]^3$ and $0$ outside of some neighborhood containing the cell. 
Then, we may interpret  $u\varphi$ as a function on~$\mathbb{R}^3$.
In all that follows, $L^p$ denotes a norm over $\mathbb{T}^3$ unless noted otherwise.
First, we note that 
  \begin{equation}
   \|D^{1/2} \omega\|_{L^8}
   \les
   \|\varphi D^{1/2} \omega\|_{L^8(\mathbb{R}^3)}
   \les
   \|D^{1/2}\curl(u\varphi)\|_{L^8(\mathbb{R}^3)}
   + \|\varphi D^{1/2} \curl u - D^{1/2} \curl (u\varphi)\|_{L^8(\mathbb{R}^3)}
   .
   \label{EQ308a}
  \end{equation} 
We focus on the  first term on the far-right side of~\eqref{EQ308a}.
By~\eqref{EQ300}, we have
  \begin{align}
   \begin{split}
   &\|D^{1/2}\curl (u\varphi)\|_{L^8(\mathbb{R}^3)} 
   \les
   \|u\varphi\|_{L^2(\mathbb{R}^3)}^{1/4} \|D\curl(u\varphi)\|_{L^\infty(\mathbb{R}^3)}^{3/4}
   \\&\indeq\indeq\les
   \|u\|_{L^2}^{1/4}
   \Big(
     \|D\omega \varphi\|_{L^\infty(\mathbb{R}^3)}^{3/4} 
     + \|Du D\varphi\|_{L^\infty(\mathbb{R}^3)}^{3/4}
     + \|\omega D\varphi\|_{L^\infty(\mathbb{R}^3)}^{3/4}
     + \|u D^2 \varphi\|_{L^\infty(\mathbb{R}^3)}^{3/4}
   \Big)
   \\&\indeq\indeq\les
   \|u\|_{L^2}^{1/4}
   \Big(
     \|D\omega\|_{L^\infty}^{3/4} 
     + \|Du \|_{L^\infty}^{3/4}
     + \|u \|_{L^\infty}^{3/4}
   \Big)
   .
   \end{split}
   \label{EQ309}
  \end{align} 
We estimate the lower order terms on the right-hand side.
Indeed, by the Gagliardo-Nirenberg inequalities, 
  \begin{equation}
   \|u\|_{L^\infty}
   \les
   \|u\|_{L^2}^{3/7}\|D^{3/2} u\|_{L^8}^{4/7} 
   +\|u\|_{L^2}^{3/7} \|u\|_{L^8}^{4/7}
   .
   \llabel{EQ310}
  \end{equation}
However, we may bound $\|u\|_{L^8} \les \|u\|_{L^\infty}$ and use Young's inequality to obtain
  \begin{equation}
   \|u\|_{L^\infty}
   \les 
   \|u\|_{L^2}^{3/7} \|D^{1/2} \omega\|_{L^8}^{4/7}
   + \|u\|_{L^2}
   \les
   \|D^{1/2}\omega\|_{L^8}
   + \|u\|_{L^2}
   .
   \label{EQ311}
  \end{equation}
Next, we consider $\|Du\|_{L^\infty}$ on the far-right side of~\eqref{EQ309}. 
Using similar reasoning as in~\eqref{EQ311}, we have
  \begin{equation}
   \|Du\|_{L^\infty}
   \les
   \|u\|_{L^2}^{1/21} \|D^{1/2} \omega\|_{L^8}^{20/21}
   + \|u\|_{L^2}
   \les
   \|D^{1/2}\omega\|_{L^8}
   + \|u\|_{L^2}
   .
   \label{EQ312}
  \end{equation}
Hence,~\eqref{EQ309},~\eqref{EQ311}, and~\eqref{EQ312} imply
  \begin{equation}
   \|D^{1/2}\curl( u\varphi)\|_{L^8(\mathbb{R}^3)}
   \les
   \|u\|_{L^2}^{1/4}\|D\omega\|_{L^\infty}^{3/4}
   + \|u\|_{L^2}^{1/4} \|D^{1/2}\omega\|_{L^8}^{3/4}
   + \|u\|_{L^2}
   .
   \label{EQ313}
  \end{equation}

Next, we consider the commutator term in~\eqref{EQ308a}.
If $\zeta$ is a smooth cut-off function equivalent to $1$ on $\supp(\varphi)$, 
we may replace $u$ in the commutator term with $\tilde{u} = u\zeta$.
Observe that
  \begin{align}
   \begin{split}
   &\|\varphi D^{1/2} \curl \tilde{u}- D^{1/2} \curl (\tilde{u}\varphi)\|_{L^8(\mathbb{R}^3)}
   \\&\indeq\indeq\les
   \|D^{1/2}\tilde{u}\|_{L^8(\mathbb{R}^3)}
   + \|\tilde u\|_{L^8(\mathbb{R}^3)}
   \les
   \|D^{1/2} u\|_{L^8}
   + \|u\|_{L^8}
   \\&\indeq\indeq\les
   \|u\|_{L^2}^{8/21} \|D^{1/2} \omega\|_{L^8}^{13/21}
   + \|u\|_{L^2}^{4/7} \|D^{1/2}\omega\|_{L^8}^{3/7}
   + \|u\|_{L^2}
   ,
   \end{split}
   \label{EQ314}
  \end{align}
where, in the last inequality, we repeated the reasoning in~\eqref{EQ311} and~\eqref{EQ312}.

Finally, we combine the estimates~\eqref{EQ308a},~\eqref{EQ313}, and~\eqref{EQ314} and use Young's inequality to attain~\eqref{EQ301}.
\end{proof}

We now prove Theorem~\ref{T02}.
\begin{proof}[Proof of Theorem~\ref{T02}]
We first consider $\Omega = \mathbb{R}^3$. 
We perform $L^p$ estimates on $D^{1/2}\omega$,
though, we will later fix $p = 8$.
Let $t \in [0,T_\ast)$.
We claim that, for $p>6$,
  \begin{equation}
   \frac{d}{dt} \|D^{1/2}\omega\|_{L^p}
   \lec
   \|D^{1/2}\omega\|_{L^p}^{1+\alpha}
   \comma
   \text{where}   
   \,\,
   \alpha = \frac{5}{6-6/p}
   ,
   \label{EQ201}
  \end{equation}
and the constant depends on $\|u_0\|_{L^2}$ and~$p$.
To see this, we differentiate the vorticity equation~\eqref{EQ103} and test
the $i$-th component of the equation with
$|D^{1/2}\omega|^{p-2}D^{1/2}\omega_i$
to obtain
  \begin{align}
   \begin{split}
    \frac{d}{dt} \|D^{1/2}\omega\|_{L^p}^p
    &\les
    \|D^{1/2}\nabla(u \omega) - u D^{1/2}\nabla\omega\|_{L^p} \|D^{1/2}\omega\|_{L^p}^{p-1}
    \\&\indeq\indeq
    + \|\omega\|_{L^\infty}\|D^{3/2}u\|_{L^p} \|D^{1/2}\omega\|_{L^p}^{p-1}
    + \|Du\|_{L^\infty} \|D^{1/2}\omega\|_{L^p}^p
    \\&\les
    \|Du\|_{L^\infty} \|D^{1/2}\omega\|_{L^p}^p
    ,
   \end{split}
    \label{EQ202}
  \end{align}
written symbolically,
where we used the Kato-Ponce inequality, the Biot-Savart law, and the bound $\|\omega\|_{L^\infty}\les \|Du\|_{L^\infty}$ in the second inequality. 
Next, we use the Gagliardo-Nirenberg inequality and the Biot-Savart law to bound
  \begin{equation}
   \|Du\|_{L^\infty} 
   \les
   \|u\|_{L^2}^{1-\alpha}\|D^{1/2}\omega\|_{L^p}^{\alpha}
   \comma 
   \text{where } 
   \alpha = \frac{5}{6- 6/p}
   .
   \label{EQ203}
  \end{equation}
Substituting this into~\eqref{EQ202} and using the conservation of energy completes the proof of the claim~\eqref{EQ201}.

Now, we derive a pointwise lower bound on~$\|D^{1/2}\omega\|_{L^p}$. 
Denote $z(t) = \|D^{1/2}\omega(t)\|_{L^p}$.
Integration of~\eqref{EQ201}
gives
  \begin{equation}
   -\frac{1}{\alpha} z(T)^{-\alpha}
   + \frac{1}{\alpha} z(t)^{-\alpha}
   \leq
   C (T - t)
   ,
   \inon{$t < T < T_\ast$,}
   \llabel{EQ206}
  \end{equation}
from which we deduce
  \begin{equation}
  z(T)^\alpha 
   \leq
   \frac{1}{z(t)^{-\alpha} - C (T-t)}
   \comma
   t < T < T_\ast
   .
   \llabel{EQ207}
  \end{equation}
Considering the denominator on the right-hand side, 
we note that $z(T)$ is bounded if $T$ satisfies
  \begin{equation}
   z(t)
   < 
   \left( \frac{1}{C (T-t)}\right)^{1/\alpha}
   .
   \label{EQ208}
  \end{equation}
However, by~\eqref{EQ203} and the bound $\|\omega\|_{L^\infty} \les \|Du\|_{L^\infty}$,
we have that $z = \|D^{1/2}\omega\|_{L^\infty}$ blows up at~$T_\ast$.
Hence, the estimate~\eqref{EQ208} must be false for $T = T_\ast$,
that is,
  \begin{equation}
   \|D\omega(t)\|_{L^p} 
   \geq
   \left( \frac{1}{C (T_\ast-t)}\right)^{1/\alpha}
   \label{EQ210}
  \end{equation}
for all $t \in [0,T_\ast)$. 

Now, we fix $p = 8$.
By Lemma~\ref{L01}, the estimate~\eqref{EQ210}, and the conservation of energy, we obtain
  \begin{equation}
   \|D\omega(t)\|_{L^\infty}
   \geq
   \frac{1}{C(T_\ast -t)^{7/5}}
   \comma
   t \in [0,T_\ast)
   ,
   \llabel{EQ211}
  \end{equation}
where we note that $4/3\alpha = 7/5$ when $p = 8$.
The proof in the case that $\Omega = \mathbb{R}^3$ is complete.

Next, we consider $\Omega = \mathbb{T}^3$.
Similarly as in~\eqref{EQ202}, we derive
  \begin{equation}
   \frac{d}{dt} \|D^{1/2}\omega\|_{L^p}
   \les
   \|D\omega\|_{L^\infty} \|D^{1/2} \omega\|_{L^p}
   .
   \label{EQ212}
  \end{equation}
The Gagliardo-Nirenberg inequality gives
  \begin{equation}
   \|Du\|_{L^\infty} 
   \les
   \|u\|_{L^2}^{1-\alpha}\|D^{1/2}\omega\|_{L^p}^{\alpha}
   + \|u\|_{L^2}
   ,
   \quad\quad
   \text{where }
   \alpha = \frac{5}{6- 6/p}
   .
   \llabel{EQ213}
  \end{equation}
Hence, the estimate~\eqref{EQ212} becomes
  \begin{equation}
   \frac{d}{dt} \|D^{1/2}\omega\|_{L^p}
   \les
   \|D^{1/2}\omega\|_{L^p} + \|D^{1/2}\omega\|_{L^p}^{1+\alpha}
   \les
   1+ \|D^{1/2}\omega\|_{L^p}^{1+\alpha}
   \les
   \left(1+ \|D^{1/2}\omega\|_{L^p}\right)^{1+\alpha}
   .
   \llabel{EQ214}
  \end{equation}
Repeating the same analysis performed for $\Omega = \mathbb{R}^3$, we deduce the lower bound 
  \begin{equation}
   \|D^{1/2}\omega(t)\|_{L^p}
   \geq 
   \left( \frac{1}{C(T_\ast -t)}\right)^{1/\alpha}
   - C
   ,
   \quad\quad
   t \in [0,T_\ast)
   ,
   \label{EQ215}
  \end{equation}
where $C$ is a constant depending on~$\|u_0\|_{L^2}$.
By Lemma~\ref{L01}, the above lower bound with $p=8$ gives
  \begin{equation}
   \|D\omega\|_{L^\infty} 
   \geq
   C\|D^{1/2}\omega\|_{L^p}^{4/3}
   - C
   .
   \llabel{EQ216}
  \end{equation}
Finally, in conjunction with~\eqref{EQ215}, we obtain
  \begin{equation}
   \|D \omega(t)\|_{L^\infty}
   \geq
   \frac{1}{C(T_\ast -t)^{7/5}} - C
   .
   \label{EQ217}
  \end{equation}
By considering $T_\ast -t < 1/C$ where $C$ depends on $\|u_0\|_{L^2}$, we may neglect the constant on the right-hand side of~\eqref{EQ217} to complete the proof. 
\end{proof}

Next, we prove lower bounds on the higher order derivatives~$\|D^k \omega\|_{L^\infty}$.
We begin with a lemma.

\begin{Lemma}
\label{L03}
Assume the hypothesis of Theorem~\ref{T03}, and let $k \geq 2$ be a positive integer.
For $\Omega = \mathbb{R}^3$ or $\mathbb{T}^3$, we have
  \begin{equation}
   \|D^{k-1} \omega\|_{L^{2(k+1)}}
   \les
   \|u\|_{L^2}^{1/(k+1)} \|D^k \omega\|_{L^\infty}^{k/(k+1)}
   .
   \llabel{EQ250}
  \end{equation}
\end{Lemma}

\begin{proof}
Define the exponents 
  \begin{equation}
    p_j
    =
    \frac{2(k+1)}{k-j}
    \comma
    j = 0,1,\dots, k-1
    ,
    \llabel{EQ251}
  \end{equation}
and $p_k = \infty$. 
By the Gagliardo-Nirenberg inequalities, we have 
  \begin{equation}
   \|D^j \omega\|_{L^{p_{j}}}
   \les
   \|D^{j-1} \omega\|_{L^{p_{j-1}}}^{1/2} \|D^{j+1} \omega\|_{L^{p_{j+1}}}^{1/2}
   \comma
   j = 1,\dots, k-1
   .
   \label{EQ252}
  \end{equation}
Also, note that 
  \begin{equation}
   \|\omega\|_{L^{p_0}}
   \les
   \|u\|_{L^2}^{1/2} \|D\omega\|_{L^{p_1}}^{1/2}
   ,
   \label{EQ253}
  \end{equation}
which is derived by integrating by parts from $\omega = \nabla \times u$ and using H\"older's inequality. 
We then replace~\eqref{EQ253} into~\eqref{EQ252} for $j=1$. 
Substituting the resulting bound into~\eqref{EQ252} for $j=2$ and repeating  this process up to $j = k-1$ completes the proof. 
\end{proof}

We now prove Theorem~\ref{T03}. 
\begin{proof}[Theorem~\ref{T03}]
We only present the proof for $\Omega = \mathbb{R}^3$
as the proof for $\mathbb{T}^3$ is analogous. 
In a similar fashion as in~\eqref{EQ201}, we have
  \begin{equation}
   \frac{d}{dt} \|D^{k-1}\omega\|_{L^p} 
   \les 
   \|Du\|_{L^\infty} \|D^{k-1}\omega\|_{L^p}
   .
   \label{EQ260}
  \end{equation}
By the Gagliardo-Nirenberg inequality and the Biot-Savart law, we have
  \begin{equation}
   \|Du\|_{L^\infty}
   \les 
   \|u\|_{L^2}^{1-\alpha} \|D^{k-1}\omega\|_{L^p}^{\alpha}
   ,
   \quad\quad
   \text{where }
   \alpha = \frac{5}{2k + 3 - 6/p}
   .
   \llabel{EQ261}
  \end{equation}
Substituting this into~\eqref{EQ260}, we have, by the conservation of energy, 
  \begin{equation}
   \frac{d}{dt} \|D^{k-1}\omega\|_{L^p}
   \les
   \|D^{k-1}\omega\|_{L^p}^{1+\alpha}
   .
   \llabel{EQ262}
  \end{equation}
As in the proof of Theorem~\ref{T02}, we derive the lower bound,
  \begin{equation}
   \|D^{k-1}\omega(t)\|_{L^p}
   \geq
   \left( \frac{1}{C(T_\ast -t)}\right)^{1/\alpha}
   \llabel{EQ263}
  \end{equation}
for all $t \in [0,T_\ast)$. 
Now, we fix $p = 2(k+1)$. By Lemma~\ref{L03}, we get
  \begin{equation}
   \|D^k \omega\|_{L^\infty}
   \ges
   \frac{1}{C(T_\ast -t)^{2k/5 + 1}}
   \comma
   t \in [0,T_\ast)
   ,
   \llabel{EQ264}
  \end{equation}  
where we note that
%$\frac{k+1}{k}\cdot \frac1\alpha = 2k/5 + 1$.
${((k+1)}/{k})(1/\alpha) = 2k/5 + 1$,
completing the proof.
\end{proof}

Finally, we prove Theorem~\ref{T04}.
\begin{proof}[Proof of Theorem~\ref{T04}]
First, we consider the case $\Omega = \mathbb{R}^3$,
for which we claim that
  \begin{equation}
   \|D^{1/2} \omega\|_{L^8} 
   \les
   \|u\|_{L^2}^{1/4}
   \|D\sigma\|_{L^\infty}^{3/4}
   .
   \label{EQ09}
  \end{equation}
To prove this inequality, we first write
  \begin{equation}
   \Vert D^{1/2}\sigma\Vert_{L^{8}}
   \lec
   \|\sigma\|_{L^4}^{1/2}
   \|D\sigma\|_{L^{\infty}}^{1/2}
   ,
   \llabel{EQ08}
  \end{equation}
noting that the lower-order terms can be scaled away.
By the Calder\'on-Zygmund theorem, we then get
  \begin{equation}
   \Vert D^{1/2}\omega\Vert_{L^{8}}
   \lec
   \|\omega\|_{L^4}^{1/2}
   \|D\sigma\|_{L^{\infty}}^{1/2}
   .
   \llabel{EQ13}
  \end{equation}
The rest of the proof proceeds exactly as the proof of
Lemma~\ref{L01},
thus leading to~\eqref{EQ09}.
Then, using \eqref{EQ09}, we invoke Theorem~\ref{T01} and the conservation of energy to get
  \begin{equation}
   \|D\sigma\|_{L^\infty} 
   \gec
   \|D^{1/2}\omega\|_{L^8}^{4/3}
   ,
   \llabel{EQ222}
  \end{equation}
and the proof for the case $\mathbb{R}^{3}$ is
complete.
For $\Omega=\mathbb{T}^{3}$, we obtain
  \begin{equation}
   \|D^{1/2} \omega\|_{L^8} 
   \les
   \|u\|_{L^2}^{1/4}
   \|D\sigma\|_{L^\infty}^{3/4}
   +
   \Vert u\Vert_{L^2}
   \llabel{EQ10}
  \end{equation}
instead of \eqref{EQ09}; the estimates
are similar to those in $\mathbb{R}^{3}$ but include lower order terms, which can be easily controlled. We omit further details.
\end{proof}

\section*{Acknowledgments}
The authors were supported in part by the NSF grant DMS-2205493.

\end{document}